\documentclass[10pt,oneside,a4paper]{amsart}
\usepackage{latexsym}
\usepackage[pdftex]{hyperref}
\usepackage{amsfonts,amsmath,amssymb,indentfirst}
\newcommand{\beq}{\begin{equation}}
\newcommand{\eeq}{\end{equation}}
\newcommand{\bdism}{\begin{displaymath}}
\newcommand{\edism}{\end{displaymath}}

\newcommand{\setN}{\mathbb N}

\newcommand{\setR}{\mathbb R}

\newtheorem{theorem}{Theorem}[section]
\newtheorem{proposition}[theorem]{Proposition}

\newtheorem{lemma}[theorem]{Lemma}
\newtheorem{remark}[theorem]{Remark}

\author{\scshape Luca Fabrizio Di Cerbo}
\title{\bf Generic properties of homogeneous Ricci solitons}
\begin{document}

\thanks{*Supported in part by the Simons Foundation.}
\address{Department of Mathematics, Duke University, Box 90320, Durham, NC 27708-0320,
USA} \email{luca@math.duke.edu}
\begin{abstract}
We discuss the geometry of homogeneous Ricci solitons. After showing
the non-existence of compact homogeneous and non-compact steady
homogeneous solitons, we concentrate on the study of left invariant
Ricci solitons. We show that, in the unimodular case, the Ricci
soliton equation does not admit solutions in the set of left
invariant vector fields. We prove that a left invariant soliton of
gradient type must be a Riemannian product with nontrivial Euclidean
de Rham factor. As an application of our results we prove that any
generalized metric Heisenberg Lie group is a non-gradient left
invariant Ricci soliton of expanding type.
\end{abstract}
\maketitle

\section{Introduction}
\pagenumbering{arabic}

A Ricci soliton is a solution of the Ricci flow
\begin{align}\label{Ricci equation}
&\frac{\partial g}{\partial t}=-2Ric_g\\ \notag & g(0)=g_0
\end{align}
that changes only by diffeomorphisms and scale. More precisely,
$g(t)$ is called Ricci soliton if there exist a smooth function
$\sigma(t)$ and a $1$-parameter family of diffeomorphisms
$\left\{\psi_{t}\right\}$ of $M^{n}$ such that
$g(t)=\sigma(t)\psi^{*}_{t}(g_{0})$ with $\sigma(0)=1$ and
$\psi_{0}=id_{M^{n}}$. It is easy to see that this condition is
equivalent to requiring that the initial metric $g_0$ satisfies the
following Einstein like identity
\begin{align}\label{lie soliton}
-2Ric_{g_{0}}=2\lambda g_{0}+ L_{X}g_{0},
\end{align}
where $\lambda$ is a real number and $X$ a complete vector field;
for the details we refer to \cite{Chow1}. We then say that the
soliton is expanding, shrinking, or steady if $\lambda>0$,
$\lambda<0$, or $\lambda=0$ respectively. Finally, if the vector
field $X$ is the gradient field of a smooth function $f$, one says
that the soliton is a gradient Ricci soliton.

The soliton theories in the compact and complete non-compact cases
are drastically different. In the compact case it is easy to prove
the non-existence of steady and expanding solitons, see
\cite{Chow1}; while a result of Perelman \cite{Per} ensures that any
shrinking soliton must be of gradient type. Moreover, there are no
solitons in dimension two and three as proved by Hamilton and Ivey
\cite{Ha1}, \cite{Ivey1}. Finally, it is interesting to notice that
nontrivial examples of compact gradient Ricci solitons were actually
constructed by Koiso in \cite{Koi}. These examples start in
dimension four and are K\"ahler. For more details about compact
K\"ahler-Ricci solitons we refer to \cite{Chow1}.

The complete non-compact theory is much richer. Besides the
shrinking and expanding Gaussian solitons on $\setR^{n}$ and the
cylinder soliton on $S^{n}\times\setR$ $(n\geq2)$ it is worthwhile
to mention the Hamilton ``cigar'' soliton on $\setR^{2}$, the
radially symmetric steady Bryant solitons on $\setR^{n}$ $(n\geq3)$
and their generalization by Ivey, who constructed  Ricci solitons on
doubly warped products, see \cite{Chow1}, \cite{Bry}, \cite{Ivey2}.
It turns out that all these examples are of gradient type.

A completely different family of solitons has been discovered by
Lauret, see \cite{Lau} and the more update \cite{Lau2}. In
\cite{Lau}, Lauret searches for a notion weakening the Einstein
condition for a left invariant metric $g$ on a nilpotent Lie group
$N^{n}$. This is motivated by the fact that nilpotent Lie groups do
not admit any left invariant Einstein metric, see \cite{Milnor}. The
Ricci soliton condition \eqref{lie soliton} clearly provides a
substitute candidate and Lauret proves that, given a nilpotent Lie
group $N^{n}$ with Lie algebra $\mathfrak{n}$ and a left invariant
metric $g$, then \eqref{lie soliton} is satisfied iff
\begin{align}\label{lauret}
Ric_{g}=cI+D,
\end{align}
for some real number $c$ and $D\in Der(\mathfrak{n})$, where
$Der(\mathfrak{n})$ denotes the Lie algebra of derivations of
$\mathfrak{n}$. Note that this condition immediately implies the
existence of a symmetric derivation. In particular, as pointed out
in \cite{Lau}, we cannot solve the soliton equation over a
characteristically nilpotent Lie group. Recall that a Lie algebra
$\mathfrak{n}$ is characteristically nilpotent iff
$Der(\mathfrak{n})$ is nilpotent. Here we just notice that this
obstruction does not apply to the class of 2-step nilpotent Lie
algebras. In fact a characteristically nilpotent Lie algebra is at
least 4-step nilpotent, for more details see \cite{Leger}. Moreover
there are no examples of characteristically nilpotent Lie algebras
in dimension less or equal that six, as follows from \cite{Moro}.

It is now interesting to consider the equation in \eqref{lauret} on
a non-nilpotent Lie group $G$, and using the $1$-parameter group of
diffeomorphisms generated by $D$ show that $g$ is a Ricci soliton.
This fact is actually used by Lauret to exhibit some examples of
non-nilpotent solvable Lie groups which admit a Ricci soliton
structure. Solitons satisfying \ref{lauret} are also called
\emph{algebraic solitons}, for more details see \cite{Lau2},
\cite{Jablonski} and the bibliography therein.

Finally, we point out that the soliton condition on a nilpotent Lie
group can be neatly characterized in terms of metric solvable
extensions.
\begin{theorem}[Lauret]\label{lauret1}
Let $N^{n}$ be a simply connected nilpotent Lie group with Lie algebra $\mathfrak{n}$ and
left invariant metric $g$. Then $(N^{n},g)$ admits a Ricci soliton structure
iff $(\mathfrak{n},g)$ admits a standard metric solvable extension
$(\mathfrak{s}=\mathfrak{a} \oplus \mathfrak{n},\,\tilde{g})$
whose corresponding simply connected Lie group $(S,\tilde{g})$ is Einstein.
\end{theorem}
Recall that a metric solvable extension of $(\mathfrak{n},g)$ is a
metric solvable Lie algebra $(\mathfrak{s}=\mathfrak{a} \oplus
\mathfrak{n},\,\tilde{g})$ such that $\left[
\mathfrak{s},\mathfrak{s}\right]
_{\mathfrak{s}}=\mathfrak{n}=\mathfrak{a}^{\perp}$, $\left[\,,\,
\right]_{\mathfrak{s}}\mid_{\mathfrak{n}\times\mathfrak{n}}=\left[\,,\,
\right]_{\mathfrak{n}}  $,
$\tilde{g}\mid_{\mathfrak{n}\times\mathfrak{n}}=g$. The metric
solvable extension is called standard if $\mathfrak{a}$ is abelian.
Although Lauret results are non-costructive, the characterization
given in Theorem \ref{lauret1} can be easily used to provide
examples. For instance, any generalized Heisenberg Lie group
(H-type) \cite{Tricerri} and many 2-step nilpotent Lie groups admit
a Ricci soliton structure, see \cite{Lau1} for a complete list of
all the known examples.

With what regard to the uniqueness of nilpotent Ricci solitons, the
following theorem provides a complete answer.

\begin{theorem}[Lauret]\label{lauret2}
Let $N^{n}$ be a simply connected nilpotent Lie group with Lie algebra $\mathfrak{n}$ and left invariant metric $g$.
If $g$ and $g^{'}$ are left invariant Ricci soliton metrics, then there exist $c>0$ and $\eta \in Aut(\mathfrak{n})$
such that $g^{'}=c\eta (g)$.
\end{theorem}

\begin{remark}
Interestingly, in a recent preprint \cite{Jablonski}, Jablonski has
extended many of the results of Lauret to the class of
solvmanifolds. Remarkably, a Ricci solvsoliton has to be an
algebraic Ricci soliton. Moreover, these solitons enjoy nice
uniqueness properties as in the nilpotent case. For more details the
reader is referred to \cite{Jablonski} and to the bibliography
therein.
\end{remark}

The first explicit construction of Lauret solitons has been obtained
by Baird and Danielo in \cite{Danielo}. In \cite{Danielo}, the
authors study $3$-dimensional Ricci solitons which projects via a
semi-conformal mapping to a  surfaces and obtain a complete
description of the soliton structures on ${\rm Nil}^{3}$ and ${\rm
Sol}^{3}$ predicted by Lauret.

In \cite{Lott}, Lott uses these particular soliton structures to
study the long time behavior of type III Ricci flow solutions and
also gives, among many other results, some explicit examples of four
dimensional homogeneous soliton solutions, e.g. on ${\rm Nil}^{4}$.
It turns out that all these solitons are expanding and of
non-gradient type. These are the first known examples of
non-gradient Ricci solitons. We also note that the property to be
non-gradient plays a important role in the analytical study of the
linear stability of these solutions, see \cite{Chr}.

In this paper, we concentrate on the study of the generic properties
of homogeneous Ricci solitons providing a new point of view on the
results found by Danielo, Baird and Lott.

\section{The scalar curvature of a homogeneous Ricci soliton}

In this section we study the evolution of the scalar curvature on a
homogeneous Ricci soliton. Recall that on a soliton solution to the
Ricci flow the scalar curvature satisfies the equation
\begin{align}\label{scalar equation}
R(t)=\frac{R_{0}}{\sigma(t)}=\frac{R_{0}}{1+2 \lambda t}.
\end{align}

In particular if the initial scalar curvature $R_{0}$ is constant it
stays constant during the evolution. In this case $R_{t}$ satisfies
the simple ODE
\begin{align}\label{Ricci equation}
&\frac{\partial R}{\partial t}=\frac{2(-\lambda)}{R_{0}}R^{2},~~ R(0)=R_0.\\ \notag
\end{align}
As proved in \cite{Ha1}, the scalar curvature of a general
Ricci flow solution evolves according the heat type equation
\begin{align}\label{evol scal}
\frac{\partial R}{\partial
t}=\Delta R+2\left|Ric\right|^{2},
\end{align}
we conclude that if the soliton has constant initial scalar
curvature with $\lambda=\frac{-R_{0}}{n}$ then it must be trivial.
In particular a nontrivial soliton structure cannot be associated to
a divergence free vector field $X$, as can be proved by considering
the trace of \eqref{lie soliton}. In summary we have:

\begin{lemma}\label{divergence}
A complete vector field $X$ associated to a nontrivial constant curvature Ricci soliton
satisfies $div(X)=k$, where $k$ is a nonzero constant.
\end{lemma}

We can now use Lemma \ref{divergence} to obtain a complete
description of the compact case.

\begin{theorem}\label{luca1}
There are no non-trivial compact solitons with constant scalar
curvature.
\end{theorem}

\begin{proof}
Stokes' theorem.
\end{proof}
With regard to the compact homogeneous case it is interesting to
notice that Theorem \ref{luca1} can also be derived from the fact
that there is a natural scaling invariant quantity that is
increasing along the flow.

\begin{lemma}\label{scaling}
Let $(M^{n},g(t))$ be a compact homogeneous solution to the Ricci
Flow, then $R(t)V(t)^{\frac{2}{n}}$ is monotonically increasing
unless the initial manifold is Einstein.
\end{lemma}
\begin{proof}
Because of the diffeomorphism invariance of the Ricci tensor and the
short time existence and uniqueness for the Ricci flow on compact
manifolds \cite{Ha2}, the flow preserves the isometries of the
initial Riemannian manifold. We conclude that a compact homogeneous
Riemannian manifold remains homogeneous during the flow. Using the
standard variation formulas for the volume and the scalar curvature
functions \cite{Ha2} we have
\begin{align}\notag
\frac{d}{dt}R(t)V(t)^{\frac{2}{n}}&=R^{'}(t)V(t)^{\frac{2}{n}}+\frac{2}{n}V(t)^{\frac{2-n}{n}}V^{'}(t)R(t)\\
\notag
&=2\left|Ric\right|^{2}V(t)^{\frac{2}{n}}-\frac{2}{n}R^{2}V(t)^{\frac{2}{n}}\\ \notag
&=2\left|Ric-\frac{R}{n}g\right|^{2}V(t)^{\frac{2}{n}}.
\end{align}
\end{proof}

Note that, because of the existence of the solitons discovered by
Lauret, there is no analogue of Lemma \ref{scaling} in the
noncompact homogeneous case. Nevertheless, we can use the
monotonicity property of the scalar curvature to derive the
following partial generalization of Theorem \ref{luca1}.

\begin{theorem}\label{luca2}
There are no steady non-compact Ricci solitons with constant scalar
curvature.
\end{theorem}
\begin{proof}
By the formula
\begin{align}
\frac{dR(t)}{dt}=2\left|Ric\right|^{2}
\end{align}
we have that the scalar curvature is increasing unless the initial metric is Ricci flat.
\end{proof}
It is interesting to note that Theorems \ref{luca1} and \ref{luca2}
immediately extend to rule out the existence of compact homogeneous
and non-compact homogeneous steady breathers. For the definition of
breather see \cite{Chow1}. This observation implies the following
remark.
\begin{remark}
As pointed out in \cite{Ha2}, the Ricci flow equation on a given
homogeneous spaces reduces to an ODE on the finite dimensional
moduli space of homogeneous metrics. Now the nonexistence of steady
breathers then implies that these interesting geometrical ODEs have
no periodic solutions, at least for those initial data that admit
geometrical interpretation. For a detailed study of the Ricci flow
ODEs on three and four dimensional homogeneous spaces we refer to
\cite{Ise1} and \cite{Ise2}.
\end{remark}
Finally, we notice that the sign of the scalar curvature determines if the soliton is expanding or shrinking.
\begin{lemma}\label{segnoc}
Any constant scalar curvature soliton of expanding or shrinking type must have respectively negative or positive initial scalar curvature.
\end{lemma}

For more results concerning homogeneous shrinking solitons with
positive scalar curvature the reader is referred to Remark
\ref{Naber} below.

\section{Left invariant Ricci solitons}

In this section we concentrate on the study of the geometrical
properties of left invariant Ricci solitons. First, we explain why
all the known explicit examples are of expanding type, then we study
the possibility to solve the soliton equation within the set of left
invariant vector fields and then we show that in ``general" a left
invariant soliton cannot be of gradient type. We point out that we
usually work on simply connected Lie group since a non-trivial
soliton structure on a Lie group with non-trivial fundamental group
can clearly be transported to its universal cover.

\subsection{The sign of the soliton constant}

As noticed in Lemma \ref{segnoc} the sign of the scalar curvature
determines the type of the soliton. The sign of the scalar curvature
associated to a left invariant metric has been extensively studied
by many authors, see in particular \cite{Milnor} and \cite{Be}. We
can then state the following.

\begin{proposition}\label{luca3}
A left invariant soliton or breather structure on a solvable Lie group is necessarily expanding.
\end{proposition}
\begin{proof}
As proved in \cite{Milnor}, any left invariant metric $g$ over a
solvable Lie group $G$ is either flat or has strictly negative
scalar curvature. If the scalar curvature negative, by Lemma
\ref{segnoc} we have that the left invariant soliton structure (if
any) must be of expanding type.
\end{proof}

This simple result explains why the examples explicitly constructed
by Baird, Danielo and Lott are of expanding type. Moreover it
ensures that every nilsoliton predicted by Lauret must be of
expanding type. The above lemma can actually be improved. In
\cite{Milnor}, Milnor studies the problem of which Lie groups
endowed with left invariant metric admit positive scalar curvature.
Using the Iwasawa decomposition theorem Wallach was able to give a
sufficient condition, namely that the universal covering of the Lie
group is not homeomorphic to an Euclidean space, for the details see
\cite{Milnor}. In \cite{Milnor} it is actually conjectured that this
condition is also necessary. This conjecture turned to be true as
shown by B\'erard-Bergery in \cite{Be}. We then have that a Lie
group whose universal cover is homeomorphic to a Euclidean space can
admit only expanding left invariant soliton structures. We also have
that a non-compact shrinking left invariant soliton must be
homeomorphic to a product of a compact Lie group with some Euclidean
space.

\begin{remark}\label{Naber}
It follows from results of Naber \cite{Naber} and Petersen-Wylie
\cite{Wylie} that any shrinking homogeneous Ricci soliton must be
trivial. More precisely, it must be isometric to a product of a
compact positive Einstein manifold with some Euclidean space. In
fact, by Theorem 1.2. in \cite{Naber} a shrinking homogeneous
soliton must be of gradient type. Finally, by Theorem 1.1. in
\cite{Wylie} such a soliton must be trivial. For more details see
also Remark \ref{yeah}.
\end{remark}

\subsection{Nonsolvability for left invariant vector fields}

We derive some generalities about the Ricci soliton equation over a
metric $g$ that is left invariant. Recall that a metric $g$ is
called left invariant if $L^{*}_{h}g=g$ for all $h\in G$ where
$L_{h}$ is the left translation by $h$. Recall also that a vector
field $X$ is called left invariant iff $L^{*}_{h}X=X$ for all $h\in
G$. Consider now the Ricci soliton equation \eqref{lie soliton} and
take the pull back by a left translation, we then get
\begin{align}\notag
&-2L^{*}_{h}Ric_{g}=L^{*}_{h}(L_{X}g)+2\lambda L^{*}_{h}g \\ \notag
&-2Ric_{L^{*}_{h}g}=L_{L^{*}_{h}X}(L^{*}_{h}g)+2\lambda g \\ \notag
&-2Ric_{g}=L_{L^{*}_{h}X}g+2\lambda g,
\end{align}
which implies \bdism L_{(L^{*}_{h}X-X)}g=0. \edism We conclude that
$L^{*}_{h}X-X$ is a Killing vector field for all $h\in G$, we then
say that $X$ is left invariant modulo Killing fields or simply a
\textit{left-Killing} field. Following this terminology a trivial
left-Killing field is just an ordinary left invariant vector field.

At this stage of the theory we cannot be sure of the existence of
nontrivial left-Killing fields, indeed one may think to consider the
usual left invariant vector fields in order to reduce the soliton
condition to an algebraic set of equations on the Lie algebra of the
group. It turns out that this is not possible in general.

\begin{lemma}\label{unimodular}
Any left invariant vector field over an unimodular Lie group is divergence free.
\end{lemma}

\begin{proof}
First, we notice that any left invariant vector field has constant
divergence. Let $X$ be a left invariant vector field over a Lie
group $ G^{n}$ equipped with a left invariant metric $g$ and let
$\left\lbrace e_{i}\right\rbrace^{n}_{i=1}$ be a left invariant
global orthonormal frame. We can then compute the divergence as
follows \bdism div(X)=\sum_{i}g(\nabla_{e_{i}}X, e_{i}) \edism and
notice that is constant since the $\left\lbrace
\nabla_{e_{i}}X\right\rbrace $ are left invariant. Recall now that a
Lie group is called unimodular if its left invariant Haar measure
$\omega$ is also right invariant.  Recall also that given a left
invariant vector field $X$ its flow can be explicitly written in
terms of right translations and the exponential map, namely for any
$x\in G^{n}$ we have \bdism F_{t}(x)=R_{\rm exp(tX)}(x). \edism Let
$d\omega$ the left invariant volume form associated to the Haar
measure. Since we are assuming $d\omega$ to be also right invariant
we clearly have $L_{X}d\omega=0$. We conclude recalling the identity
$L_{X}d\omega=div(X)d\omega$.
\end{proof}

We can now combine the above result with Lemma \ref{divergence} to
obtain the following proposition.

\begin{proposition}\label{nonsolvability}
For any unimodular Lie group the homogeneous soliton equation cannot
be solved within the set of left invariant vector fields.
\end{proposition}

Note that Proposition \ref{nonsolvability} implies the existence of
nontrivial left-Killing fields. Recall that, in terms of the Lie
algebra $\mathfrak{g}$, the unimodular condition is equivalent to
requiring that the linear transformation $ad\;x$ has trace zero for
every $x\in\mathfrak{g}$. We conclude that any nilpotent Lie group
is unimodular which implies, by Proposition \ref{nonsolvability},
the existence of at least one complete nontrivial left-Killing field
on any nilsoliton.

We now briefly study the same question on non-unimodular Lie groups.
Notice that on a non-unimodular Lie group the divergence of a left
invariant vector field can be different from zero. Let
$(\mathfrak{g},\left\langle, \right\rangle )$ be a three dimensional
non-unimodular metric Lie algebra. As shown in Section 6 of
\cite{Milnor}, we can always find an orthonormal basis $e_{1}$,
$e_{2}$, $e_{3}$ so that
\begin{align}\notag
&\left[e_{1}, e_{2}\right]=\alpha e_{2}+\beta e_{3} \\ \notag
&\left[e_{1}, e_{3}\right]=\gamma e_{2}+\delta e_{3} \\ \notag
\end{align}
with $\left[e_{2}, e_{3}\right]=0$, $\alpha+\delta\neq 0$ and
$\alpha\gamma+\beta\delta=0$. Following the literature we refer to
this special frame as a Milnor frame (M-frame). The divergence of
the elements in the M-frame is easily computed:
\begin{align}\notag
div(e_{1})&=g(\nabla_{e_{1}}e_{1},e_{1})+g(\nabla_{e_{2}}e_{1},e_{2})+g(\nabla_{e_{3}}e_{1},e_{3}) \\ \notag
&=0+\frac{1}{2}\left\lbrace g(\left[e_{2}, e_{1}\right], e_{2})-g(\left[e_{1}, e_{2}\right], e_{2})+g(\left[e_{2}, e_{2}\right], e_{1})  \right\rbrace \\ \notag
&+\frac{1}{2}\left\lbrace g(\left[e_{3}, e_{1}\right], e_{3})-g(\left[e_{1}, e_{3}\right], e_{3})+g(\left[e_{3}, e_{3}\right], e_{1})  \right\rbrace \\ \notag
&=-(\alpha+\delta);
\end{align}
an analogous compututation shows that $div(e_{2})=div(e_{3})=0$. In
summary for any orthonormal M-frame $e_{1}$ has nonzero divergence
while $e_{2}$ and $e_{3}$ are divergence free. Assume now the
soliton equation can be solved within the set of left invariant
vector fields. Let $(\mathfrak{g},\left\langle , \right\rangle )$
our soliton metric Lie algebra, and let $e_{1}$, $e_{2}$, $e_{3}$
the associated M-frame. Notice that the left invariant vector field
that solve the soliton condition must have a nonzero component in
the direction of $e_{1}$. Using Lemma 6.5. in \cite{Milnor}, it is
easy to express the Ricci tensor in terms of the constants of
structure of the Lie algebra. Using this fact and performing
analogous computations for the Lie derivative of the metric we
reduce the soliton condition $-2Ric_{g}=2\lambda g+L_{X}g$, where
$X=ae_{1}+be_{2}+ce_{3}$ with $a\neq 0$, to the following set of
algebraic equations
\begin{gather}
-2\left( \begin{array}{ccc}
 -\alpha^{2}-\delta^{2}-\frac{1}{2}(\beta+\gamma)^{2}& 0 & 0 \\
0 & -\alpha(\alpha+\delta)+\frac{1}{2}(\gamma^{2}-\beta^{2})& 0 \\
0 & 0 & -\delta(\alpha+\delta)+\frac{1}{2}(\beta^{2}-\gamma^{2})
\end{array} \right) \nonumber \\
= 2\lambda
\left( \begin{array}{ccc}
1 & 0 & 0 \\
0 & 1 & 0 \\
0 & 0 &1
\end{array} \right) +a
\left( \begin{array}{ccc}
0 & 0 & 0 \\
0 & -2\alpha & -(\gamma+\beta) \\
0 & -(\gamma+\beta) &-2\delta
\end{array} \right)\nonumber \\ +\left( \begin{array}{ccc}
0 & b\alpha+c\delta & b\beta+c\delta \\
b\alpha+c\delta & 0 & 0 \\
b\beta+c\delta & 0 & 0
\end{array} \right) .
\nonumber
\end{gather}
This system is easily reduced to
\begin{align}\notag
&2(\alpha^{2}+\delta^{2})=2\lambda \\ \notag
&2\alpha(\alpha+\delta)=2\lambda-2a\alpha \\ \notag
&2\delta(\alpha+\delta)=2\lambda-2a\delta.
\end{align}
Thus, $\lambda=\alpha^{2}+\delta^{2}$ which implies, under the
assumption $\delta\neq \alpha$, $a=-(\alpha+\delta)$. Substituting
this value in the second of the equations above we obtain
$\delta=\alpha=0$ that is a contradiction. In the remaining case
$\delta=\alpha$ an easy computation shows that $a=0$ which clearly
implies the triviality of the soliton. In summary we proved the
following proposition.

\begin{proposition}
For any three dimensional non-unimodular Lie algebra the soliton
equation cannot be solved within the set of left invariant vector
fields.
\end{proposition}

We don't know if a similar restriction holds in dimension greater
than three.

\subsection{Left invariant gradient solitons}

As mentioned in the introduction, all the known explicit
constructions of left invariant Ricci solitons are of non-gradient
type. Remarkably these are the first examples of non-gradient Ricci
solitons. We now show that this property is indeed generic for this
class of solitons.

Recall that the gradient Ricci soliton equation is given by
\begin{align}\label{gradient ricci}
Ric_{g}+\nabla\nabla f+\dfrac{\lambda}{2}g=0.
\end{align}
We then say that $g$ is an expanding, shrinking, or steady soliton if $\lambda>0$, $\lambda<0$ or $\lambda=0$ respectively.
By standard manipulations in \eqref{gradient ricci} we can derive the following useful identities
\begin{align}\label{trace1}
R+\Delta f+\dfrac{\lambda}{2}n=0
\end{align}
\begin{align}\label{trace2}
dR=2Ric(\nabla f,\cdot)
\end{align}
\begin{align}\label{trace3}
\Delta R=g(\nabla R,\nabla f)-\lambda R-2\left|Ric\right|^{2},
\end{align}
for a proof see \cite{Chow1}. Using identity \eqref{trace1}, we
notice that if $R$ is constant then $\Delta f$ is constant. We can
then derive another proof of theorem \ref{luca1}. In fact, by
Perelman's result on compact Ricci solitons we can restrict our
attention to shrinking gradient solitons. Thus, using the
nonexistence of sub/superharmonic functions on compact manifolds we
can rule out again the constant scalar curvature case. Combining now
identities \eqref{trace1} and \eqref{trace3}, we have
\begin{align}
\dfrac{2}{n}(R+\Delta f)R=2\left|Ric\right|^{2}\notag
\end{align}
that implies
\begin{align}
\dfrac{2}{n}\Delta f R=2\left|Ric-\dfrac{R}{n}g\right|^{2}.
\end{align}
We conclude that on a gradient shrinking soliton with constant
scalar curvature $\Delta f$ and $R$ have the same sign. Finally,
using \eqref{trace2} we can derive a nice rigidity property for
constant scalar curvature solitons. Indeed, using the identity
\begin{align}\label{bochner}
0=dR=2Ric(\nabla f,\cdot),
\end{align}
we have that if the Ricci curvature has definite sign then the
soliton must necessarily be Einstein. More precisely we can state
the following:

\begin{proposition}\label{degenerate}
A Riemannian manifold with constant scalar curvature and non degenerate Ricci tensor cannot be a gradient Ricci soliton.
\end{proposition}

We conclude our study showing that Proposition \ref{degenerate} can
actually be improved if we restrict to left invariant solitons.
Formula \eqref{bochner} combined with the Bochner identity for
$\left|\nabla f\right|^{2}$ gives
\begin{align}\notag
\Delta\left|\nabla f\right|^{2}&=2\left|\nabla \nabla f\right|^{2}+2Ric(\nabla f,\nabla f)+2g(\nabla f,\nabla\Delta f) \\ \notag
&=2\left|\nabla \nabla f\right|^{2},
\end{align}
then if we further assume $\left|\nabla f\right|^{2}=c$, with $c$ a constant, we derive that the soliton is trivial.
Thus, since
\begin{align}
L^{*}_{h}\nabla_{g}f=\nabla_{L^{*}_{h}g}(f\circ L_{h})=\nabla_{g}(f\circ L_{h}),
\end{align}
we have that for at least some $h\in G$ the vector field
$\nabla_{g}(f\circ L_{h}-f)$ is a nontrivial Killing field. Now, a
gradient Killing field must be parallel. We conclude that our group
splits locally as a Riemannian product with an interval. The simple
observation that a gradient soliton structure on Lie group lifts to
its universal Riemannian cover immediately implies the following
theorem.

\begin{theorem}\label{irreducible}
Any metric Lie algebra $(\mathfrak{g},\,\left\langle, \right\rangle )$ with trivial Euclidean de Rham factor cannot
 be a gradient Ricci soliton.
\end{theorem}

\begin{remark}\label{yeah}
It follows from a rigidity result of Petersen-Wylie \cite{Wylie}
that any homogeneous gradient Ricci soliton must have a non-trivial
Euclidean deRham factor. The reader should compare Theorem
\ref{irreducible} with the more general Theorem 1.1. in
\cite{Wylie}. The first version of this work and the preprint of
\cite{Wylie} appeared independently on the arXiv in October 2007.
\end{remark}

We can now apply Proposition \ref{degenerate} to the class of
nonsingular 2-step nilpotent Lie algebras.

\begin{theorem}\label{nonsingular}
A nonsingular 2-step nilpotent Lie algebra cannot admit a left invariant gradient soliton structure.
\end{theorem}

\begin{proof}
Recall that a 2-step nilpotent Lie algebra $\mathfrak{n}$ is called
nonsingular if the map $ad\;x:\mathfrak{n}\rightarrow \mathfrak{z}$
is surjective for all $x \in\mathfrak{n}-\mathfrak{z}$, where
$\mathfrak{z}$ is the center of $\mathfrak{n}$.  Now, on a general
2-step nilpotent Lie algebra equipped with a positive definite inner
product $\left\langle , \right\rangle$, denoted by $\mathfrak{v}$
the orthogonal complement of $\mathfrak{z}$ in $\mathfrak{n}$, we
can define for each $z\in\mathfrak{z}$ a skew symmetric linear
transformation $j(z):\mathfrak{v}\rightarrow\mathfrak{v}$ by \bdism
j(z)x=(ad\;x)^{*}z \; {\rm for}\; {\rm all} \;x\in\mathfrak{v}
\edism where $(ad\;x)^{*}$ denotes the adjoint of $ad\;x$. As
extensively shown by many authors, see in particular \cite{Eber} and
\cite{Tricerri}, most of the geometry of the metric Lie algebra
$\left\lbrace\mathfrak{n}, \left\langle , \right\rangle\right\rbrace
$ is given by the properties of the maps $j(z)$. In particular the
kernel of the Ricci tensor, seen as a symmetric linear
transformation on the Lie algebra, is given by the following
subspace of the center $\left\lbrace z\in \mathfrak{v}\;\lvert\;
j(z)=0\right\rbrace $, see Proposition 2.5 in \cite{Eber}. Being
$\mathfrak{n}$ nonsingular, an easy argument shows that
$j(z):\mathfrak{v}\rightarrow\mathfrak{v}$ is nonsingular for any
nonzero $z\in\mathfrak{z}$. We conclude that for any nonsingular
2-step nilpotent metric Lie algebra $\left\lbrace\mathfrak{n},
\left\langle , \right\rangle\right\rbrace $ the associated Ricci
tensor is non-degenerate. The claim is now a consequence of Theorem
\ref{irreducible}.
\end{proof}

We then have the following theorem for generalized Heisenberg Lie
groups (H-type).

\begin{theorem}\label{gabriele}
Any simply connected metric H-type Lie group admits an expanding
non-gradient left invariant Ricci soliton structure.
\end{theorem}
\begin{proof}
Recall that a 2-step nilpotent metric lie algebra
$\left\lbrace\mathfrak{n}, \left\langle , \right\rangle\right\rbrace
$ is of Heisenberg generalized type if $j(z)^{2} =-\lvert   z
\rvert^{2}id_{\mathfrak{v}}$ for all $z\in\mathfrak{z}$. This
clearly implies that for any $z\in\mathfrak{z}$ the linear map
$j(z)$ is nonsingular. Now a routine linear algebra argument shows
that the Lie algebra $\mathfrak{n}$ must be nonsingular. By Theorem
\ref{nonsingular} we conclude that a H-type Lie group cannot be a
gradient Ricci soliton. For what regard the existence we follow
Lauret. The standard metric solvable extension of the H-type Lie
groups are exactly the harmonic Damek-Ricci spaces that are solvable
and Einstein, see for example \cite{Tricerri}. By Theorem
\ref{lauret1} we conclude that these groups are algebraic Ricci
solitons. Finally, using Proposition \ref{luca3} we have that these
structures must be of expanding type.
\end{proof}

\section{Final remarks}

Theorem \ref{gabriele} provides the first example of a countably
infinite family of non-gradient Ricci solitons. Recall that for any
$n\in\setN$ there exist a countably infinite number of
non-isomorphic H-type Lie algebras with $dim\mathfrak{z}=n$, see
again \cite{Tricerri}. It is also interesting to notice that any
H-type group admits a lattice, i.e., a cocompact discrete subroup.
This result is easily derived from the existence of an integral
structure on any H-type Lie algebra, see \cite{Grandall}. In fact a
well known result by Malcev, see for example \cite{Rag}, ensures
that a simply connected nilpotent Lie group admits a lattice iff its
Lie algebra admits a rational structure. Thus, let $N$ be a H-type
Lie group and $\Gamma$ a lattice. We  then have that the
left-Killing vector field $X$ is not preserved by the action of the
subroup $\Gamma$ otherwise we could project the soliton structure on
$(N,g)$ to a well defined homogeneous soliton structure on the
compact manifold $(N/\Gamma,\tilde{g})$. Thus, Theorem
\ref{gabriele} also provides an infinite family of
\textit{pseudosolitons}, i.e., compact manifold with no soliton
structure that acquire one when lifted to the universal cover, see
\cite{Chr} for the definition of pseudosolitons.

As a final application we derive that the 3-dimensional non-product
geometries ${\rm Nil}^{3}$, ${\rm Sol}^{3}$ and ${\rm SL}(2,\setR)$
do not admit any left invariant gradient soliton structure, since
any left invariant metric on these groups is de Rham indecomposable,
see \cite{Scott}.  Note how a similar result cannot be obtained by
using Proposition \ref{degenerate}. In fact, as proved in
\cite{Milnor}, the signature of the Ricci form on both ${\rm
Sol}^{3}$ and ${\rm SL}(2,\setR)$ can be either $(+,-,-)$ or
$(0,0,-)$. More in general Theorem \ref{irreducible} provides a
satisfactory obstruction in the case of irreducible left invariant
solitons.

It is interesting to note that very recently Jablonski in
\cite{Jablonski} has finally rule out the existence of a homogeneous
Ricci soliton structure on ${\rm SL}(2,\setR)$. This result, in a
sense, closes up the connection between the theory of homogeneous
Ricci solitons and the special geometries of Thurston.

Finally, I would like to point out that since the appearance of
first version of this work on the arXiv, the theory of homogeneous
Ricci solitons has vigorously grown. The interested reader should
for example consult \cite{Glick}, \cite{Jablonski0}, \cite{Payne},
\cite{Lau1}, \cite{Ise2}, \cite{Jablonski}, \cite{Wylie},
\cite{Naber}, \cite{Lau2}.

\vskip 5mm

\noindent\textbf{Acknowledgements}. I would like to thank Jorge
Lauret for his kind interest in this work and for several
constructive comments. Moreover, I would like to thank Professor
John Milnor for several useful discussions regarding \cite{Milnor}.
I also would like to thank the referee for guiding me through the
most recent literature on homogeneous solutions of the Ricci flow
and for pertinent comments on this manuscript.

\end{document}